\documentclass[a4paper,12pt]{amsart}

\usepackage{graphicx}

\newtheorem{theorem}{Theorem}[section]
\newtheorem{proposition}[theorem]{Proposition}

\newtheorem{lemma}[theorem]{Lemma}

\newtheorem{Example}[theorem]{Example}

\newtheorem{Remark}[theorem]{Remark}

\newtheorem{Remarks}[theorem]{Remarks}

\newtheorem{Question}[theorem]{Question}

\newcommand{\al}{\alpha}

\newcommand{\de}{\delta}
\newcommand{\eps}{\varepsilon}

\newcommand{\La}{\Lambda}

\newcommand{\si}{\sigma}

\let\cal=\mathcal
\let\Bbb=\mathbb

\begin{document}

\title{Ikehara-type theorem involving boundedness}


\date{July 3, 2008}

\author{Jacob Korevaar}

\begin{abstract}
Let $\sum a_n/n^z$ be a Dirichlet series with nonnegative
coefficients that converges to a sum function
$f(z)=f(x+iy)$ for $x>1$. Setting
$s_N=\sum_{n\le N}\,a_n$, the paper gives a necessary and
sufficient condition for boundedness of $s_N/N$. As
$x\searrow 1$, the quotient $f(x+iy)/(x+iy)$ must converge
to a pseudomeasure $q(1+iy)$, the distributional Fourier
transform of a bounded function. The paper also gives an
optimal estimate for $s_N/N$ under the `real condition'
$f(x)=\cal{O}\{1/(1-x)\}$. 
\end{abstract}

\maketitle

\setcounter{equation}{0}    
\section{Introduction} \label{sec:1}
We recall the famous Tauberian theorem of Ikehara:
\begin{theorem} \label{the:1.1}
Suppose that the Dirichlet series
\begin{equation} \label{eq:1.1}
\sum_{n=1}^\infty\frac{a_n}{n^z}\quad\mbox{with}\;\;a_n\ge
0\;\;\mbox{and}\;\;z=x+iy  
\end{equation}
converges throughout the half-plane $\{x>1\}$, so that the sum
function $f(z)$ is analytic there.
Suppose furthermore that there is a constant $A$ such that the 
difference
\begin{equation} \label{eq:1.2}
g(z)=f(z)-\frac{A}{z-1}
\end{equation}
has an analytic or continuous extension to the closed half-plane
$\{x\ge 1\}$. Then the partial sums $s_N=\sum_{n\le N}a_n$
satisfy the limit relation
\begin{equation} \label{eq:1.3}
s_N/N\to A \quad\mbox{as}\;\;N\to\infty.
\end{equation}
\end{theorem}
The theorem is often called `Wiener--Ikehara theorem'
because Ikehara studied with Wiener, and applied the new
Tauberian method that Wiener was developing in the years
1926--1931; see \cite{Ik31}, \cite{Wi32}, \cite{Wi33} and cf.\
\cite{Ko02}, \cite{Ko04a}. Ikehara's theorem led to a greatly
simplified proof of the prime number theorem. 

In \cite{Ko04b} the author obtained a two-way form of the
theorem. Given $f(z)$ of the form (\ref{eq:1.1}), the following
condition is necessary and sufficient for (\ref{eq:1.3}): The
difference $g(z)=g(x+iy)$ must have a distributional limit
$g(1+iy)$ for $x\searrow 1$ which is locally equal to a
pseudofunction. That is, on every finite interval $(-B,B)$, the
distribution $g(1+iy)$ must be equal to a pseudofunction which
may depend on $B$. A pseudofunction is the distributional
Fourier transform of a bounded function that tends to $0$ at
infinity. It can also be characterized as a tempered
distribution which is locally given by Fourier series with
coefficients that tend to $0$. A pseudofunction may have
nonintegrable singularities, but not as strong as first-order
poles.
 
In connection with Ikehara's theorem one may ask (cf.\ Mhaskar
\cite{Mh08}) what condition on $f(z)$ would suffice for the
conclusion that
\begin{equation} \label{eq:1.4}
s_N/N=\cal{O}(1)\quad\mbox{as}\;\;N\to\infty.
\end{equation}
Unlike the situation in the case of power series $\sum a_nz^n$,
it is not enough when $f(\cdot)$ satisfies the `real condition'
\begin{equation} \label{eq:1.5}
f(x)=\cal{O}\{1/(x-1)\}\quad\mbox{as}\;\;x\searrow 1.
\end{equation}
\begin{proposition} \label{prop:1.2}
For Dirichlet series $(\ref{eq:1.1})$ with sum $f(z)$, condition
$(\ref{eq:1.5})$ implies the estimate
\begin{equation} \label{eq:1.6}
s_N/N=\cal{O}(\log N),
\end{equation}
and this order-estimate is best possible.
\end{proposition}
See Section \ref{sec:2}. In Section \ref{sec:3} we
will prove
\begin{theorem} \label{the:1.3} 
Let the series $\sum a_n/n^z$ with coefficients
$a_n\ge 0$ converge to $f(z)=f(x+iy)$ for $x>1$. Setting
$s_N=\sum_{n\le N}\,a_n$ as before, the sequence $\{s_N/N\}$
will remain bounded if and only if the quotient 
\begin{equation} \label{eq:1.7}
q(x+iy)=\frac{f(x+iy)}{x+iy}\qquad(x>1)
\end{equation}
converges in the sense of tempered distributions to a
pseudomeasure $q(1+iy)$ as $x\searrow 1$. 
\end{theorem}
A pseudomeasure is the distributional Fourier transform of a
bounded measurable function. It has local representations by
Fourier series with uniformly bounded coefficients. A simple
example is given by the delta distribution or Dirac measure. The
following pseudomeasure is the boundary distribution of an
analytic function: 
$$\frac{1}{+0+iy}\stackrel{\mathrm{def}}{=}\lim_{x\searrow
0}\,\frac{1}{x+iy}=
\lim_{x\searrow 0}\,\int_0^\infty e^{-xt}e^{-iyt}dt.$$ 
It is the Fourier transform of the Heaviside function $1_+(t)$,
which equals $1$ for $t\ge 0$ and $0$ for $t<0$. Pseudomeasures
can have no singularities worse than first-order poles; cf.\ 
(\ref{eq:3.2}) below.

\setcounter{equation}{0} 
\section{Proof of Proposition \ref{prop:1.2}} \label{sec:2}
The proof consists of two parts.

(i) Let $f(z)=\sum a_n/n^z$ with $a_n\ge 0$ as in (\ref{eq:1.1})
satisfy the real condition (\ref{eq:1.5}). Setting
$x=x_N=1+1/\log N$, one finds that 
$$\si_N \stackrel{\mathrm{def}}{=}\sum_{n\le
N}\,\frac{a_n}{n}\le e\sum_{n\le
N}\,\frac{a_n}{n}e^{-(\log n)/\log N}
\le ef(x_N)=\cal{O}(\log N).$$  
A crude estimate now gives the result of (\ref{eq:1.6}):
$$s_N = \sum_1^N\,n\,\frac{a_n}{n}\le N\si_N=\cal{O}(N\log
N).$$

(ii) For the second part we use an example. 
\begin{lemma} \label{lem:2.1} Let 
\begin{equation} \label{eq:2.1}
a_n =\left\{\begin{array}{ll}
2^{2^k+k} & \mbox{for
$n=2^{2^k},\;k=1,\,2,\,\cdots$},\\ 0 & \mbox{for
all other $n$.}
\end{array}\right.
\end{equation}
Then 
\begin{equation} \label{eq:2.2}
f(x)=\sum\,\frac{a_n}{n^x}
=\cal{O}\{1/(x-1)\}\quad\mbox{as}\;\;x\searrow 1,
\end{equation}
but  
\begin{equation} \label{eq:2.3}
\mbox{for}\;\;N=2^{2^k},\;\;\mbox{one
has}\;\;s_N\ge a_N=(1/\log 2)N\log N.
\end{equation}
\end{lemma}
\begin{proof}
Take $x=1+\de$ with $0<\de<1$. Then
$$f(x)=\sum\frac{2^{2^k+k}}{2^{2^kx}}=\sum\frac{2^k}
{2^{2^k\de}}.$$
Observe that the graph of
$$h(t)=\frac{2^t}{2^{2^t\de}}\qquad(0<t<\infty)$$
is rising to a maximum at some point $t=t_0(\de)$ and
then falling. Thus the sum for $f(x)$ is majorized by the
integral of $h(t)$ over $(0,\infty)$ plus the value $h(t_0)$.
Both have the form $const/\de$, hence (\ref{eq:2.2}).

Now take $N$ of the form $2^{2^k}$, so that $\log N=2^k\log 2$.
Then
$$a_N=2^{2^k+k}=N2^k = N(\log N)/\log 2.$$
\end{proof}

\setcounter{equation}{0} 
\section{Proof of Theorem \ref{the:1.3}} \label{sec:3}
Note that the distributional convergence in the theorem is
convergence in the Schwartz space $\cal{S}'$. In other words,  
\begin{equation} \label{eq:3.1}
<q(x+iy),\phi(y)>\;\to\;
<q(1+iy),\phi(y)>\quad\mbox{as}\;\;x\searrow 1
\end{equation}
for all testing functions $\phi(y)\in\cal{S}$, that is, all
rapidly decreasing $C^\infty$ functions; see Schwartz
\cite{Sch66} or H\"{o}rmander \cite{Ho83}.  

\begin{proof}[Proof of Theorem \ref{the:1.3}] Let $f(z)$
be the sum of the Dirichlet series in the theorem. Now define
$s(v)=\sum_{n\le v}\,a_n$, so that $s(v)=0$ for $v<0$ and
$s_N=s(N)=\cal{O}(N^{1+\eps})$ for every $\eps>0$. Integrating
by parts, one obtains a representation for
$q(z)=f(z)/z$ as a Mellin transform:
$$q(z)=(1/z)\int_{1-}^\infty v^{-z}ds(v) =
\int_1^\infty s(v)v^{-z-1}dv\qquad(x>1).$$
The substitution $v=e^t$ gives $q(z)$ as a shifted
Laplace transform of $S(t)=e^{-t}s(e^t)$:
$$q(z) = \int_0^\infty s(e^t)e^{-zt}dt=\int_0^\infty
S(t)e^{-(z-1)t}dt\qquad(x>1).$$

\smallskip
(i) Suppose that the sequence $\{s_N/N\}$ is bounded. Then
$S(t)$ is bounded, $|S(t)|\le M$, say. Hence
\begin{equation} \label{eq:3.2}
|q(z)|\le\frac{M}{x-1}\quad\mbox{for}\;\;x>1.
\end{equation}
Thus the boundary singularities of $q(z)$ on the line $\{x=1\}$
can be no worse than first-order poles. We will verify that in
the sense of distributions,
$$q(x+iy)\to q(1+iy)\stackrel{\mathrm{def}}{=}\hat S(y),$$
where $\hat S(y)$ denotes the distributional Fourier transform
of $S(t)$. Indeed, for fixed $x>1$, the function $q(x+iy)$ is
the Fourier transform of
$S_x(t)=S(t)e^{-(x-1)t}$. Since $|S(t)|\le M$ and
$S(t)=0$ for $t<0$, the functions $S_x(t)$ converge to
$S(t)$ boundedly as $x\searrow 1$, hence in the sense of
tempered distributions. Since distributional Fourier
transformation is continuous on the Schwartz space $\cal{S}'$,
it follows that $q(x+iy)$ converges to the distributional
Fourier transform of $S(t)$ -- in this case a pseudomeasure. 

\smallskip
(ii) Conversely, suppose that $q(x+iy)=\hat S_x(y)$ converges to
a pseudomeasure as $x\searrow 1$, symbolically written as
$q(1+iy)$. Then $q(1+iy)$ is the Fourier transform $\hat H(y)$
of a bounded function $H(t)$. By the continuity of inverse
Fourier transformation, this implies that $H(t)$ is the
distributional limit of $S_x(t)=S(t)e^{-(x-1)t}$ as $x\searrow
1$. But the latter limit is equal to $S(t)$: 
$$<S_x(t),\phi_0(t)>\,=\int_0^\infty
S_x(t)\phi_0(t)dt\;\to\;<S(t),\phi_0(t)>$$
for all $C^\infty$ functions $\phi_0(t)$ of compact support. It
follows that $S(t)=H(t)$ on ${\Bbb R}$, hence bounded.
\end{proof}

\setcounter{equation}{0} 
\section{Final remarks} \label{sec:4}
Let $\pi_2(N)$ denote the number of prime twins $(p,\,p+2)$
with $p\le N$. The famous twin-prime conjecture (TPC) of Hardy
and Littlewood \cite{HL23} asserts that for $N\to\infty$,
\begin{equation} \label{eq:4.1}
\pi_2(N)\sim 2C_2{\rm li}_2(N)=
2C_2\int_2^N\frac{dt}{\log^2 t}\sim
2C_2\frac{N}{\log^2 N}.
\end{equation}
Here $C_2$ is the `twin-prime constant',
$$C_2 = \prod_{p\,{\rm
prime},\,p>2}\,\left\{1-\frac{1}{(p-1)^2}\right\}
\approx 0.6601618.$$
For the discussion of the TPC it is convenient to introduce the
modified counting function
\begin{equation} \label{eq:4.2}
\psi_2(N)\stackrel{\mathrm{def}}{=}\sum_{n\le
N}\,\La(n)\La(n+2),
\end{equation}
where $\La(k)$ denotes von Mangoldt's function. Since
$\La(k)=\log p$ if $k=p^\al$ for some prime number $p$ and
$\La(k)=0$ otherwise, the TPC turns out to be
equivalent to the asymptotic relation
\begin{equation} \label{eq:4.3}
\psi_2(N)\sim 2C_2N\quad\mbox{as}\;\;N\to\infty.
\end{equation}
It is natural then to introduce the Dirichlet series
\begin{equation} \label{eq:4.4}
D_2(z)\stackrel{\mathrm{def}}{=}\sum_{n=1}^\infty
\,\frac{\La(n)\La(n+2)}{n^z}\qquad(z=x+iy,\;x>1).
\end{equation}
By a sieving argument, cf.\ Halberstam and Richert \cite{HR74},
one has $\pi_2(N)=\cal{O}(N/\log^2N)$, or equivalently,
$\psi_2(N)=\cal{O}(N)$. By Theorem \ref{the:1.3} another
equivalent statement is that the quotient $D_2(x+iy)/(x+iy)$
converges distributionally to a pseudomeasure as $x\searrow 1$.
And finally, by the two-way Ikehara--Wiener theorem referred to
in Section \ref{sec:1}, the TPC is equivalent to the
conjecture that the difference $D_2(z)-2C_2/(z-1)$ has local
pseudofunction boundary behavior as $x\searrow 1$.

\bigskip

\noindent{\scshape KdV Institute of Mathematics, University of 
Amsterdam, \\
Plantage Muidergracht 24, 1018 TV Amsterdam, Netherlands}

\noindent{\it E-mail address}: {\tt J.Korevaar@uva.nl}

\enddocument